\documentclass{article}
\usepackage{amsmath,amsfonts,amssymb,amsthm} %amsthm can't be here or it conflicts with the siamltex pre-programed settings
\usepackage{graphicx}
\usepackage[T1]{fontenc}
\usepackage[utf8]{inputenc}
\usepackage{multirow}
\usepackage{verbatim}
\usepackage{multicol}

\newtheorem{theorem}{Theorem}[section]

\newtheorem{conjecture}[theorem]{Conjecture}
\newtheorem{example}{Example}[section]
\newtheorem{lemma}[theorem]{Lemma}

\theoremstyle{definition}

\DeclareMathOperator{\td}{\bf{td}}

\title{Classes of critical graphs for tree-depth}
\author{Michael D. Barrus\thanks{Department of Mathematics, University of Rhode Island, Kingston, Rhode Island 02881, United States; email: \texttt{barrus@uri.edu}} ~and
John Sinkovic\thanks{Department of Mathematics and Statistics, Georgia State University, Atlanta, Georgia 30302, United States; email: \texttt{johnsinkovic@gmail.com}}}

\begin{document}

\maketitle

\begin{abstract}
A $k$-ranking of a graph $G$ is a labeling of the vertices of $G$ with values from $\{1,\dots,k\}$ such that any path joining two vertices with the same label contains a vertex having a higher label. The tree-depth of $G$ is the smallest value of $k$ for which a $k$-ranking of $G$ exists. The graph $G$ is $k$-critical if it has tree-depth $k$ and any proper minor of $G$ has smaller tree-depth, and it is 1-unique if for every vertex $v$ in $G$, there exists an optimal ranking of $G$ in which $v$ is the unique vertex with label 1.

We present several classes of graphs that are both $k$-critical and $1$-unique, providing examples that satisfy conjectures on critical graphs discussed in [M.D.~Barrus and J.~Sinkovic, Uniqueness and minimal obstructions for tree-depth, submitted].

\end{abstract}

%%%%%%%%%%%%%%%%%%%%%%%%%%%%%%%%%%%%%%%%%%%%%%%%%%%%%%%%%%%%%%%%%%%%%%%%%%%%%%%%%%%%%%%%%%%%%

\section{Introduction}
The \emph{tree-depth} of a graph $G$, denoted $\td(G)$, is the smallest natural number $k$ such that each vertex of $G$ may be labeled with an element from $\{1,\dots,k\}$ so that every path joining two vertices with the same label contains a vertex having a larger label.
A \emph{$k$-ranking} of $G$ is such a labeling. We say that $G$ is ($k$-)\emph{critical} if it has tree-depth $k$ and any proper minor of $G$ has smaller tree-depth, and it is \emph{1-unique} if for every vertex $v$ in $G$, there exists an optimal ranking of $G$ in which $v$ is the unique vertex with label 1.

The definition of 1-uniqueness was introduced in~\cite{BarrusSinkovic15}, motivated in part by the following two conjectures:

\begin{conjecture}[\cite{DGT}] \label{conj: critical graph orders}
	Every critical graph with tree-depth $k$ has at most $2^{k-1}$ vertices.
\end{conjecture}

\begin{conjecture} \label{conj: max degree}
	Every critical graph with tree-depth $k$ has maximum degree at most $k-1$.
\end{conjecture}

The paper~\cite{BarrusSinkovic15} presents the following additional conjecture and shows that it implies Conjecture~\ref{conj: max degree}:

\begin{conjecture} \label{conj: 1-unique}
	All critical graphs are 1-unique.
\end{conjecture}

Results in~\cite{BarrusSinkovic15} also show how critical, 1-unique graphs may be used to inductively construct larger critical graphs with a prescribed tree-depth. In order to begin, it is necessary to have identified some critical, 1-unique graphs. The purpose of this note is to present a few classes of such graphs.

In the following, let $V(G)$ denote the vertex set of a graph $G$. The \emph{order} of $G$ is given by $|V(G)|$. Given a vertex $v$ of $G$, let $G-v$ denote the graph resulting from the deletion of $v$. Similarly, given a set $S \subseteq V(G)$ or an edge $e$ of $G$, let $G-S$ and $G-e$ respectively denote the graph obtained by deleting all vertices in $S$ or by deleting edge $e$ from $G$. We indicate the disjoint union of graphs $G$ and $H$ by $G+H$, and we indicate a disjoint union of $k$ copies of $G$ by $kG$. We use $\langle p_1,\dots,p_k\rangle$ to denote a path from $p_1$ to $p_k$, with vertices listed in the order the path visits them; the \emph{length} of such a path is $k-1$, the number of its edges. We use $K_n$, $P_n$, and $C_n$ to denote the complete graph, path, and cycle with $n$ vertices, respectively.

\section{Families of $1$-unique Critical Graphs} \label{sec:families}

We now present some families of graphs that are both critical and $1$-unique. We recall a few results from~\cite{BarrusSinkovic15}, as well as facts about tree-depths of paths and cycles.

A vertex $v$ in a graph $G$ is \emph{1-unique} if there exists an optimal ranking of $G$ where $v$ is the unique vertex with rank 1.  If $v$ is any vertex in a graph $G$, a \emph{star-clique transform at $v$} removes $v$ from $G$ and adds edges as necessary so that $N_G(v)$ becomes a clique.

\begin{theorem}[\cite{BarrusSinkovic15}]\label{thm:starclique} Let $v$ be a vertex of a graph $G$, and let $H$ be the graph obtained through the star-clique transform at $v$ of $G$.  Vertex $v$ is $1$-unique in $G$ if and only if $\td(H)<\td(G)$.
\end{theorem}

\begin{theorem}[\cite{BarrusSinkovic15}]\label{thm:conjecture} If $G$ is 1-unique and has tree-depth $k$, and deleting any edge of $G$ results in a graph with a lower tree-depth, then $G$ is $k$-critical.
\end{theorem}

\begin{lemma}[\cite{path,cycle}]\label{lem:pathscycles}
For $n \geq 1$, \[\td(P_n)=\lfloor \log_2 n\rfloor +1,\] and for $n \geq 3$, 
\[\td(C_n)=\lfloor \log_2 (n-1)\rfloor +2.\]
\end{lemma}

\begin{lemma} \label{lem:families}
For all integers $k \geq 1$, the graphs $K_k$, $C_{2^{k-2}+1}$, and $P_{2^{k-1}}$ are both $k$-critical and $1$-unique.
\end{lemma}

\begin{proof}
Fix a natural number $k$. It is easy to see that $K_k$ is $k$-critical.  Since each vertex receives a different color in any ranking, by symmetry each vertex is $1$-unique.

Looking at the values of $\td(P_n)$ and $\td(C_n)$ in Lemma~\ref{lem:pathscycles}, we see that deleting or contracting any edge of $C_{2^{k-2}+1}$ or $P_{2^{k-1}}$ lowers its tree-depth from $k$.  It follows from Theorem~\ref{thm:starclique} that a vertex of degree 1 or 2 in a critical graph is $1$-unique. Thus $C_{2^{k-2}+1}$ and $P_{2^{k-1}}$ are $1$-unique.
\end{proof}

We now present another family of 1-unique critical graphs. For each $k \geq 3$ and $t$ such that $0 \leq t \leq 2^{k-2}-2$, let $R_{k,t}$ be the graph obtained by taking a path with $2^{k-2}+1+t$ vertices and adding an edge between the two vertices at distance $t$ from the endpoints. As illustrated in Figure \ref{fig:Umt}, $R_{k,t}$ contains a cycle with $2^{k-2}+1-t$ vertices and two attached paths of length $t$; note that $R_{k,0}=C_{2^{k-2}+1}$.

\begin{figure}[htb]
       \center{\includegraphics[scale=0.6]{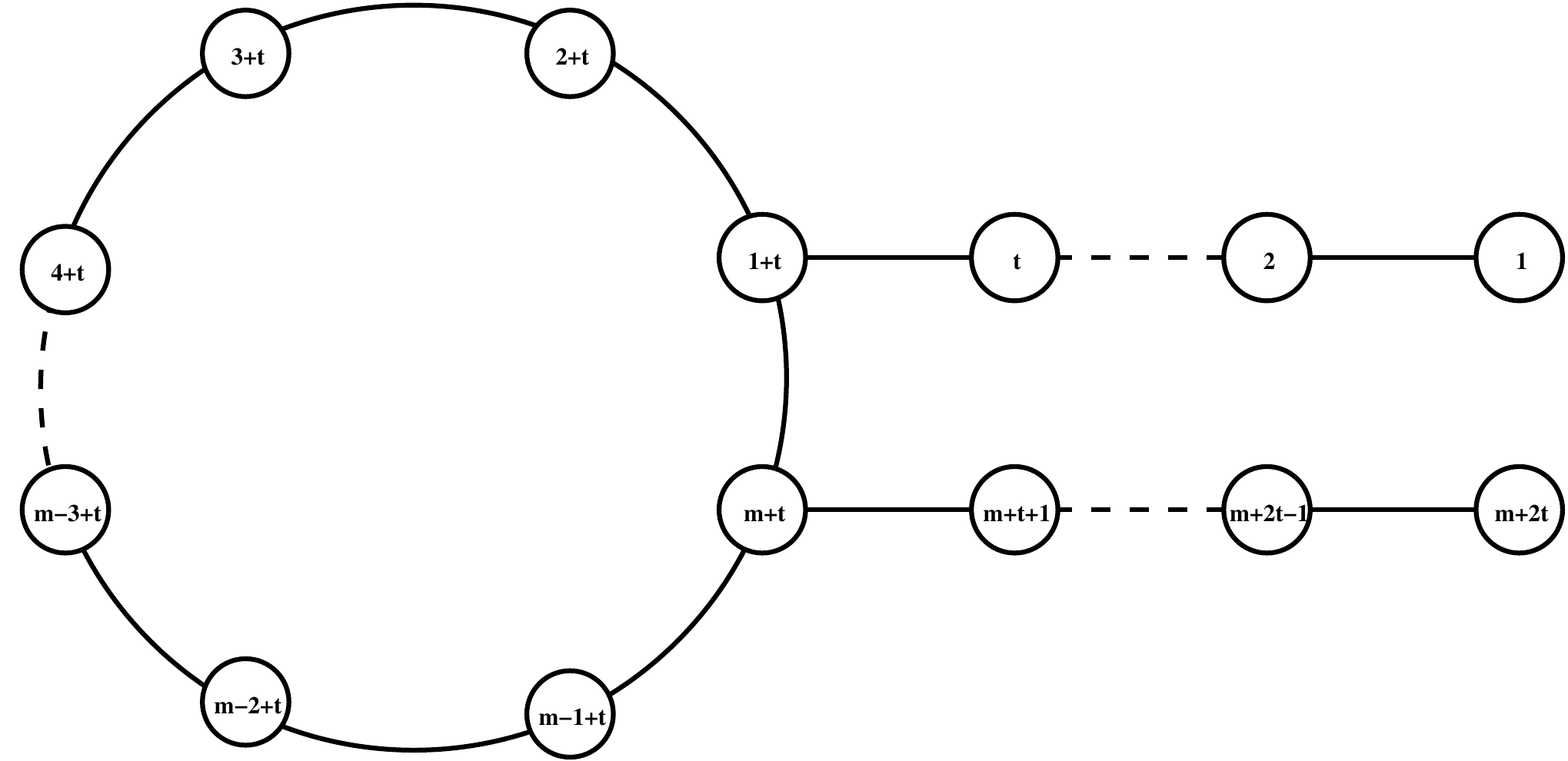}}
       \caption{\label{fig:Umt} The graph $R_{k,t}$}
\end{figure}

\begin{lemma}
For integers $k \geq 3$ and $t$ such that $0 \leq t \leq 2^{k-2}-2$, the graph $R_{k,t}$ is $k$-critical and $1$-unique.
\end{lemma}
\begin{proof}
Let $m=2^{k-2}+1-t$.  Then $R_{k,t}$ can be thought as a cycle of order $m$ with two pendant paths of length $t$.  By Lemma~\ref{lem:families} it suffices to assume that $t\geq 1$ and hence $m \leq 2^{k-2}$.

We show that $\td(R_{k,t})\geq k$ by showing that deleting any vertex $x$ leaves a subgraph with tree-depth at least $k-1$. Now $R_{k,t}$ has a unique path of order $m+2t$. Denote its vertices $1,\ldots, m+2t$ in order.  If $x$ is not a vertex of degree 2 on the cycle, then $x\in \{1,\ldots, t+1\}\cup\{m+t,\ldots, m+2t\}$.  Thus $R_{k,t}-x$ has a path of order $t+m-1=2^{k-2}$, which has tree-depth $k-1$.  Suppose henceforth that $t+1<x<m+t$.  If $t+1\geq 2^{k-3}$,  then $R_{k,t}-x$ contains a path of order $2(t+1)\geq 2^{k-2}$, so we also assume that $t+1<2^{k-3}$.

Let $p=1+\lfloor \log_2(t+1)\rfloor$; then $2^{p-1}\leq t+1<2^p$.  We will exhibit a collection of $2^{k-2-p}$ paths of order $2^p$, none of which contain $x$, that together induce a tree with tree-depth $k-1$.

If $x\equiv a \mod 2^p$ where $1\leq a\leq t+1$, then we begin the collection with the two paths $\langle 1,\ldots, 2^p\rangle$ and $\langle m+2t-2^p+1. \dots, m+2t\rangle$.   Suppose that we add to this collection $i$ paths of order $2^p$ covering vertices $2^p+1,\ldots, 2^p+i2^p$ and $2^{k-2-p}-2-i$ paths of order $2^p$ covering vertices $m+2t-2^p-(2^{k-2-p}-2-i)2^p+1,\ldots, m+2t-2^p$. Since $m+2t-2^p-(2^{k-2-p}-2-i)2^p+1 = t+2 + 2^p + i2^p$, the vertices of the cycle not covered by one of these paths are $1+(i+1)2^p, \ldots, t+1+(i+1)2^p$.  Clearly we may choose $i$ so that $x$ is one of these uncovered vertices.

If $x\equiv a \mod 2^p$ where $t+2\leq a\leq 2^p$, then we begin our collection of paths of order $2^p$ with the path $\langle 1,\ldots, t+1, m+t, \ldots, m+2^p-2\rangle$.   Suppose to the collection we add $i$ paths of order $2^p$ covering vertices $t+2,\ldots, t+1+i2^p$ and $2^{k-2-p}-1-i$ paths of order $2^p$ covering vertices $m+t-(2^{k-2-p}-1-i)2^p,\ldots, m+t-1$.  Since $m+t-(2^{k-2-p}-1-i)2^p = 1+2^p+i2^p$, the vertices of the cycle not covered by one of these paths are $t+2+i2^p,\ldots, 2^p+i2^p$.  Clearly we may choose $i$ so that $x$ is one of these uncovered vertices.

Let $\mathcal{C}$ denote our collection of $2^{k-2-p}$ paths of order $2^p$ that avoid $x$. By~\cite[Lemma 2]{DGT}, for any $t \geq 1$, adding a single edge joining two $t$-critical graphs creates a $(t+1)$-critical graph. Note that each path in $\mathcal{C}$ is joined to one or two other paths in $\mathcal{C}$ by edges of the cycle from $R_{k,t}$; let $\mathcal{E}$ denote the edges in $R_{k,t}-x$ that join vertices in distinct elements of $\mathcal{C}$. Since the cycle in $R_{k,t}$ is reduced to a path in $R_{k,t} - x$, these cycle edges naturally assign a linear order to the paths in $\mathcal{C}$. Joining the first two of these paths with the appropriate edge from $\mathcal{E}$ creates a tree with tree-depth $p+2$, as does joining the third and fourth, fifth and sixth, and so on. Further adding the edge from $\mathcal{E}$ joining the second and third (sixth and seventh, etc.) trees creates an induced subgraph with tree-depth $p+3$. Continuing in this way, because the number of paths in $\mathcal{C}$ is $2^{k-2-p}$, adding to the paths in $\mathcal{C}$ all edges from $\mathcal{E}$ results in an induced subgraph of $R_{k,t}-x$ with tree-depth $(p+1)+(k-2-p) = k-1$, as claimed. Thus $\td(R_{k,t}) \geq k$.

Deleting a vertex of degree 3 yields two paths of orders $m+t-1$ and $t$.  By our assumptions on $m$ and the definition of $t$, we know that $m+t-1=2^{k-2}$ and $t\leq 2^{k-2}-2$, so these two paths have tree-depth less than or equal to $k-1$. Thus $\td(R_{k,t})\leq k$.

If $x$ is any vertex, then $x$ lies on a path of order $2^{k-2}$ when the farther vertex of degree 3 (call it $y$) is deleted.   By Lemma~\ref{lem:families}, there is a ranking of $P_{2^{k-2}}$ with $k-1$ colors such that vertex $x$ is the unique vertex with color 1.  The other path in $R_{k,t}-x$ has tree-depth at most $k-2$, so it has a ranking using colors $2,\ldots, k-1$.
Use the colors on these paths for the corresponding vertices of $R_{k,t}$, and  assign color $k$ to the vertex $y$.   The result is an optimal ranking of $R_{k,t}$ where $x$ is the unique vertex colored $1$.  Hence $R_{k.t}$ is $1$-unique.

To show that $R_{k,t}$ is critical, by Theorem~\ref{thm:conjecture} it suffices to show that the tree-depth decreases when any edge is deleted.   Deleting any edge $e$ on the cycle of $R_{k,t}$ leaves a tree.  By~\cite[Theorem 2]{DGT} and induction, all $k$-critical trees have order $2^{k-1}$.   Since $R_{k,t}-e$ has order $m+2t=2^{k-1}+2-m$ and $m\geq 3$, $\td(R_{k,t}-e)<k$.  Deleting an edge not on the cycle yields two components, each of which is a proper induced subgraph of $R_{k,t}$.  Since $R_{k,t}$ is $1$-unique, deleting any vertex of $R_{k,t}$ results in a graph with a lower tree-depth (since we may obtain a $(k-1)$-ranking by taking a $k$-ranking where the chosen vertex is the unique vertex receiving label 1, deleting that vertex, and dropping the labels on all remaining vertices by 1). Hence any proper subgraph has tree-depth less than $k$; it follows that $\td(R_{k,t}-e)<k$.
\end{proof}

Before describing our final class of 1-unique critical graphs, we recall a theorem from~\cite{BarrusSinkovic15} that will be used many times. Given a graph $G$ and a subset $S$ of its vertex set, let $G\langle S \rangle$ denote the graph with vertex set $S$ in which vertices $u$ and $v$ are adjacent if they are adjacent in $G$ or if some component of $G-S$ has a vertex adjacent to $u$ and a vertex adjacent to $v$.

\begin{theorem}[\cite{BarrusSinkovic15}]\label{thm:newbound}
If $G$ is a graph, then \[\td(G) = \min_{S \subseteq V(G)} \left(\td(G\langle S \rangle) + \td(G-S)\right).\] Furthermore, $\td(G) = \td(G\langle T \rangle) + \td(G-T)$ if and only if there exists an optimal ranking of $G$ in which the vertices in $T$ receive higher colors than the vertices outside $T$.
\end{theorem}

Consider now a graph $Q$ constructed in the following way: Given $k \geq 1$ and $s \in \{1,\dots,k\}$, let $H_0$ be a complete graph with $s$ vertices, and for some $q \in \{1,\dots,s\}$, and let $H_1, \dots, H_q$ be vertex-disjoint complete graphs, each with $k-s$ vertices. Given a partition $\pi_1+\dots+\pi_q$ of $s$ into positive integers, choose a partition $B_1,\dots,B_q$ of the vertices of $H_0$ so that $|B_i|=\pi_i$ for all $i$, and form $Q$ by adding to the disjoint union $H_0+H_1+\dots+H_q$ all possible edges between vertices in $B_i$ and vertices of $H_i$, for all $i$.

Call the class of all graphs built in this way $\mathcal{Q}_{k,s}$. Note that $\mathcal{Q}_{k,s}$ has as many nonisomorphic graphs as there are integer partitions of $s$.

\begin{example} \rm Since there are seven integer partitions of 5, the class $\mathcal{Q}_{7,5}$ contains seven graphs. We illustrate two of them in Figure~\ref{fig:Q7,5}. There the three pairs of outermost vertices in each graph comprise the vertex sets of $H_1, H_2, H_3$, while the five innermost vertices form $B_1\cup B_2\cup B_3$.  For the graph on the left, $|B_1|=|B_2|=2$ and $|B_3|=1$.  For the graph on the right, $|B_1|=|B_2|=1$ and $|B_3|=3$. %
\begin{figure}[htb]
       \center{\includegraphics[scale=0.3]{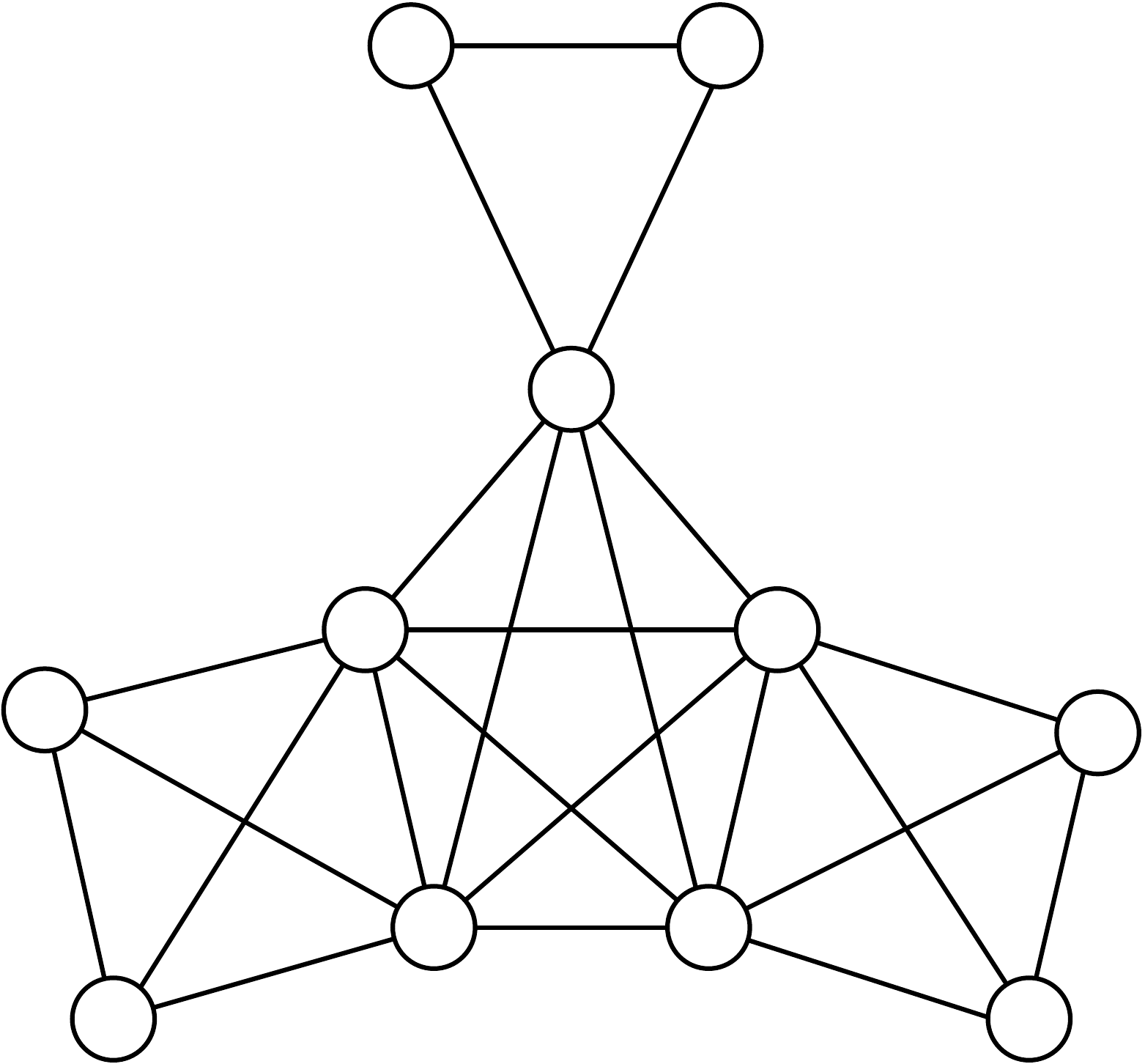} \hspace{.5cm} \includegraphics[scale=0.3]{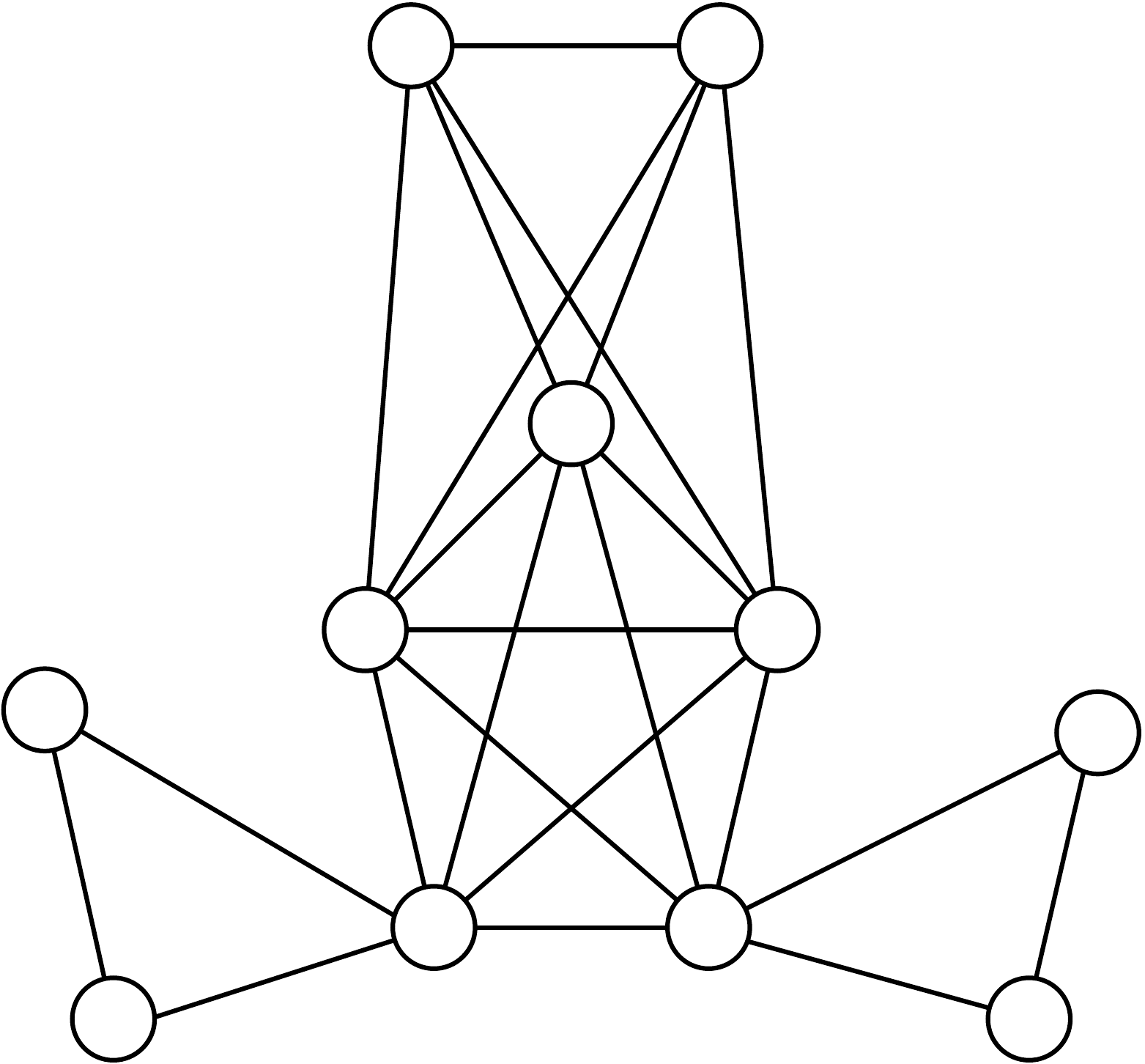}}
       \caption{\label{fig:Q7,5} Two graphs in $\mathcal{Q}_{7,5}$}
\end{figure}
\phantom{.}\hfill $ \Box$
\end{example}

We now show that, as in these examples, graphs in $\mathcal{Q}_{k,s}$ are critical.

\begin{lemma}
Every graph in $\mathcal{Q}_{k,s}$ is $k$-critical and $1$-unique.
\end{lemma}
\begin{proof}
The claim was established for the case $s=k$ in Lemma~\ref{lem:families}, so assume that $s<k$. We prove that $\td(Q) = k$ for every graph $Q$ in $\mathcal{Q}_{k,s}$ by induction on $k$. When $k = 2$, the only graph in $\mathcal{Q}_{k,s}$ is $K_2$, which by Lemma~\ref{lem:families} is $2$-critical and $1$-unique. Now suppose the claim is true for the families $\mathcal{Q}_{k',s'}$ for all values of $k'$ less than $k$. Let $Q$ be an arbitrary element of $\mathcal{Q}_{k,s}$. We show that $Q$ is $k$-critical and $1$-unique.

Let $q$ and $H_i,B_i$ ($1 \leq i \leq q$) be the parameter and sets involved in the construction of $Q$ as described above. To see that $\td(Q) \leq k$, we observe from Theorem~\ref{thm:newbound} that \[\td(Q) \leq \td\left(Q \left\langle \bigcup_i B_i \right\rangle\right) + \td\left(Q-\bigcup_i B_i\right) = \td(K_s) + \td(qK_{k-s}) = k.\]

To see that $\td(Q) \geq k$, we show that deleting any vertex from $Q$ leaves a graph having a subgraph with tree-depth at least $k-1$. If $v$ is a vertex from $H_j$ for some $j$, then let $Q'$ be the graph resulting from deleting a vertex from every clique $H_i$, with $v$ being the vertex deleted from $H_j$. Note that $Q'$ is a subgraph of $Q-v$ that belongs to $\mathcal{Q}_{r-1,s}$. If $v$ is a vertex from $B_j$, then $Q-v$ has a component $Q''$ that belongs to $\mathcal{Q}_{k,s-1}$ (or is $K_{k-1}$, if $s=1$). In both cases $Q'$ and $Q''$ have tree-depth $k-1$ by the induction hypothesis or Lemma~\ref{lem:families}.

With $Q$ again representing an element of $\mathcal{Q}_{k,s}$, we now show that every vertex of $Q$ is $1$-unique. Let $v$ be a vertex of $Q$, and let $R$ be the graph resulting from a star-clique transform on $Q$ at $v$. We show that $\td(R)<k$.

If $v$ belongs to some set $H_j$ in $Q$, then $R$ is isomorphic to $Q-v$. Let $w$ be a vertex in $B_j$. By Theorem~\ref{thm:newbound},
\begin{multline*}
\td(R) \leq \td\left(R \left\langle \bigcup_i B_i - \{w\} \right\rangle\right) + \td\left(R - \left(\bigcup_i B_i - \{w\}\right)\right)\\
= \td\left(K_{s-1}\right) + \td\left(qK_{k-s}\right) = k-1.
\end{multline*}

If $v$ instead belongs to $B_j$ in $Q$, then in $R$ the vertex $v$ is deleted and all possible edges are added between $H_j$ and $B_i$ for $i=1,\dots,q$. By Theorem~\ref{thm:newbound},
\begin{multline*}
\td(R) \leq \td\left(R \left\langle \bigcup_{i \neq j} B_i\right\rangle\right) + \td\left(R - \left(\bigcup_{i \neq j} B_i\right)\right)\\
= \td\left(K_{s-|B_j|}\right) + \td\left(K_{|B_j|+k-s-1}+(q-1)K_{k-s}\right) = k-1.
\end{multline*}

Thus $R$ has tree-depth at most $k-1$, and it follows from Theorem~\ref{thm:starclique} that $G$ is 1-unique. We complete the proof by showing that $Q$ is $k$-critical. By Theorem~\ref{thm:conjecture}, it suffices to show that for any edge $e$ of $Q$, deleting $e$ lowers the tree-depth. Let $Q'=Q-e$.

If $e$ joins two vertices in $H_j \cup B_j$ for some $j$, then by Theorem~\ref{thm:newbound},
\begin{multline*}
\td(Q') \leq \td\left(Q' \left\langle \bigcup_{i \neq j} B_i\right\rangle\right) + \td\left( Q' - \bigcup_{i \neq j} B_i\right)\\
= \td\left(K_{s-|B_j|}\right) + \td\left((K_{|B_j|+k-s}-e)+(q-1)K_{k-s}\right) = k-1.
\end{multline*}

If $e=uv$ where $u$ and $v$ belong to distinct sets $B_j$ and $B_{j'}$, then by Theorem~\ref{thm:newbound},
\begin{multline*}
\td(Q') \leq \td\left(Q' \left\langle \bigcup_{i} B_i\right\rangle\right) + \td\left(Q' - \bigcup_{i \neq j} B_i\right)\\
= \td\left(K_{s}-e\right) + \td\left(qK_{k-s}\right)
= k-1.
\end{multline*}
\end{proof}

\bibliographystyle{elsarticle-num}
\bibliography{treedepth}

\begin{thebibliography}{1}
\expandafter\ifx\csname url\endcsname\relax
  \def\url#1{\texttt{#1}}\fi
\expandafter\ifx\csname urlprefix\endcsname\relax\def\urlprefix{URL }\fi
\expandafter\ifx\csname href\endcsname\relax
  \def\href#1#2{#2} \def\path#1{#1}\fi

\bibitem{BarrusSinkovic15}
M.~D. Barrus, J.~Sinkovic, Uniqueness and minimal obstructions for tree-depth.
  Submitted.

\bibitem{DGT}
Z.~Dvo{\v{r}}{\'a}k, A.~C. Giannopoulou, D.~M. Thilikos,
  \href{http://dx.doi.org/10.1016/j.ejc.2011.09.014}{Forbidden graphs for
  tree-depth}, European J. Combin. 33~(5) (2012) 969--979.
\newblock \href {http://dx.doi.org/10.1016/j.ejc.2011.09.014}
  {\path{doi:10.1016/j.ejc.2011.09.014}}.
\newline\urlprefix\url{http://dx.doi.org/10.1016/j.ejc.2011.09.014}

\bibitem{path}
M.~Katchalski, W.~McCuaig, S.~Seager,
  \href{http://dx.doi.org/10.1016/0012-365X(93)E0216-Q}{Ordered colourings},
  Discrete Math. 142~(1-3) (1995) 141--154.
\newblock \href {http://dx.doi.org/10.1016/0012-365X(93)E0216-Q}
  {\path{doi:10.1016/0012-365X(93)E0216-Q}}.
\newline\urlprefix\url{http://dx.doi.org/10.1016/0012-365X(93)E0216-Q}

\bibitem{cycle}
E.~Bruoth, M.~Hor{\v{n}}{\'a}k,
  \href{http://dx.doi.org/10.7151/dmgt.1094}{On-line ranking number for cycles
  and paths}, Discuss. Math. Graph Theory 19~(2) (1999) 175--197, the Seventh
  Workshop ``3in1'' Graphs '98 (Krynica).
\newblock \href {http://dx.doi.org/10.7151/dmgt.1094}
  {\path{doi:10.7151/dmgt.1094}}.
\newline\urlprefix\url{http://dx.doi.org/10.7151/dmgt.1094}

\end{thebibliography}

\end{document}